\documentclass{amsart}
\usepackage{stmaryrd} 
\usepackage{hyperref}
\theoremstyle{plain}
\newtheorem{theorem}{Theorem}

\newtheorem{corollary}[theorem]{Corollary}
\newtheorem*{main}{THEOREM}
\newtheorem{lemma}[theorem]{Lemma}
\newtheorem{proposition}[theorem]{Proposition}

\theoremstyle{definition}
\newtheorem{example}{Example}

\newcommand{\mL}{\mathbb L}
 \newcommand{\mH}{\mathbb H}
 \newcommand{\E}{\mathbb E}
\newcommand{\PP}{\mathbb P}
 \newcommand{\N}{\mathbb N}
 \newcommand{\R}{\mathbb R}
 \newcommand{\mS}{\mathbb S}
 \newcommand{\T}{\mathbb T}
\newcommand{\Z}{\mathbb Z}
\newcommand{\cA}{\mathcal A}
\newcommand{\cC}{\mathcal C}

 \newcommand{\cM}{\mathcal M}
 
\newcommand{\cP}{\mathcal P}
 
\newcommand{\cS}{\mathcal S}

 \newcommand{\tH}{\tilde{H}}
\newcommand{\tL}{\tilde{L}}

\newcommand{\xb}{\bar{x}}
\newcommand{\Xb}{\bar{X}}

\newcommand{\tb}{\bar{t}}

\newcommand{\tcM}{\widetilde{\mathcal M}}
\newcommand{\tcA}{\tilde{\mathcal A}}
\newcommand{\gab}{\bar{\gamma}}
\newcommand{\lab}{\bar{\lambda}}

\newcommand{\af}{\alpha}

 \newcommand{\ga}{\gamma}
\newcommand{\dga}{\dot{\gamma}}
 \newcommand{\Ga}{\Gamma}
 \newcommand{\de}{\delta}

\newcommand{\ep}{\varepsilon}
 \newcommand{\lam}{\lambda}
 \newcommand{\om}{\omega}
 \newcommand{\Om}{\Omega}
 \newcommand{\te}{\theta}

\newcommand{\lb}{\llbracket}
\newcommand{\rb}{\rrbracket}

\newcommand{\entre}{\setminus}

\newcommand{\sop}{\operatorname{supp}}

\begin{document}
\title[Vanishing viscosity for periodic Hamilton-Jacobi equations]
{Vanishing viscosity limits for space-time periodic 
Hamilton-Jacobi equations}
\author[E. Guerra]{Eddaly Guerra}
\address{ Instituto de Matem\'aticas, UNAM. Ciudad Universitaria C. P. 
          04510, Cd. de M\'exico, M\'exico.}
\email{eddaly@yahoo.com.mx}
\author[H. S\'anchez-Morgado]{H\'ector S\'anchez-Morgado}
 \email{hector@math.unam.mx}
\begin{abstract}
Extending previuos results (\cite{AIPS},\cite{Be}), we study the vanishing 
viscosity limit of solutions of space-time periodic
Hamilton-Jacobi-Belllman equations, assuming that the ``Aubry set''
is the union  of a finite number of hyperbolic periodic orbits of the
Hamiltonian flow.  
\end{abstract}
 \maketitle

\section{Introduction}
 Let $M$ be a compact manifold and
 $H:T^*M\times\R\to\R$ be a $C^k$ Hamiltonian $k\ge 3$, 
satisfying the standard hypotheses 
 \begin{description}\label{hypo}
 \item [Convexity] The Hessian $H_{pp}(x,p,t)$ is positive definite.
\item [Superlinearity]
 $$
 \lim_{|p|\to \infty}\frac{H(x,p,t)}{|p|}=\infty,
 $$
 uniformly on $x,t$.
\item [Periodicity] The Hamitonian is also time periodic
 $$H(x,p,t+1)=H(x,p,t),$$
 for all $x,p,t$.
\item [Completeness] The Hamitonian flow $\phi^*_t$ in $T^*M\times\mS^1$
is complete.
\end{description}
Let $L:TM\times\R\to\R$ be the Lagrangian associated to the 
Hamiltonian: 
\[ L(x,v,t)=\max_ppv-H(x,p,t).\] 
For our main result we will assume that $M$ is the $d$-dimensional
torus $\T^d$ and the Hamiltonanin satisfies the following  growth condition:

There is $K>0$ such that for all $x,p,t$ with $|p|\ge K$  
\begin{equation}\label{growth}
(H_p\cdot p-H+\inf_{(x,t)}H(x,0,t))K-|H_x|\ge 0.
\end{equation}
A natural example that satisfies the hypotheses is given by
\[H(x,p)=\frac 12|p+P|^2+V(x,t)\] 
with $V\in C^k(\T^{d+1})$, since in this case the flow
is complete and hypothesis \eqref{growth} reads
 \[(\frac 12|p|^2-\frac 12|P|^2-V(x,t)+\inf_{(x,t)}V(x,t))K-|V_x(x,t)|\ge 0\]
for $|p|\ge K$.

For $c\in\R$ consider the Hamilton-Jacobi equation 
\begin{equation}\label{HJT}
u_t+H(x,Du,t)=c.
\end{equation}
It is know (\cite{CIS}) that there is only one value $c=c(L)$, the so
called critical value, such that \eqref{HJT} has a time periodic
viscosity solution. 

Consider also the Hamilton-Jacobi-Bellman equation
\begin{equation}\label{HJB}
\phi_t+\ep\Delta\phi(x,t)+H(x, D\phi(x,t),t) =c(\ep)
\end{equation}
As for equation \eqref{HJT}, there is only one constant
$c(\ep )$ such that \eqref{HJB} admits solutions. 
However, this time the solution is unique
up to an additive constant (\cite{bar}), and we denote it $\phi_\ep$.
We study the behaviour of $\phi_\ep$ as $\ep$ goes to zero.
We prove that the family
$(\phi_\ep)_{\ep>0}$ is equi-Lipschitz, so that we can extract
subsequences which converge uniformly (see Lemma \ref{lip}). 
By the stability theorem for viscosity solutions (\cite{CEL84},
\cite{Barles}, \cite{Bardi}), limits as $\ep\to 0$ of such
subsequences have to be viscosity solutions of equation \eqref{HJT}.

We shall assume that the ``Aubry set'' (section \ref{sec:preliminar}) of
the Lagrangian system, is the union of a finite number of hyperbolic
periodic orbits of the Euler-Lagrange flow.
Extending previuos results (\cite{AIPS},\cite{Be}), we describe the
limits in terms of the orbits in the Aubry set that minimize,
the normalized integral along the orbit, of the Laplacian  of the
corresponding ``Peierls barrier'' (section \ref{sec:preliminar}).
In particular we prove that the limit is unique if there is only 
one orbit in the Aubry set that minimizes that normalized integral.

\section{Preliminaries and statement of main result}
\label{sec:preliminar}
Define the action of an absolutely  continuous curve  $\ga:[a,b]\to M$ as
 $$A_L(\ga):=\int_{a}^{b}L(\ga(\tau),\dot\ga(\tau),\tau) d\tau $$
A curve $\ga:[a,b]\to M$ is {\em closed } if
$\ga(a)=\ga(b)$ and $b-a\in\Z$.  The critical value can be defined as
\[c(L):= \min \{k\in\R:\forall\ga\text{ closed }A_{L+k}(\ga) \ge 0\}\]
Let ${\cP}(L)$ be the set of probabilities on the Borel $\sigma$-algebra 
of $TM\times\mS^1$ that have compact support and are invariant under
the Euler-Lagrange flow. Then
\[c(L)= -\min \{\int Ld\mu : \mu\in\cP(L)\}.\]
The Mather set is defined as 
\[\tcM:=
\overline{\bigcup\{\sop\mu:\mu\in\cP(L),\int Ld\mu=-c(L)\}}.\]
 For $a\le b$, $x,y\in M$ let $\cC(x,a,y,b)$ be the set 
 of absolutely continuous curves $\ga:[a,b]\to M$ with $\ga(a)=x$ and
 $\ga(b)=y$. 
For $a\le b$ define $F_{a,b}:M\times M\to\R$ by  
  $$F_{a,b}(x,y):= \min\{A_L(\ga):\ga\in\cC(x,a,y,b)\}.$$

For $t\in\R$ let $[t]$ be the corresponding point in
${\mS}^1$ and $\lb t\rb$ be the integer part of $t$.
 Define the {\em action potential}
$\Phi:M\times\mS^1\times M\times\mS^1\to\R$ by  
 \[\Phi(x,[s],y,[t]):=\min\{F_{a,b}(x,y)+c(L)(b-a):[a]=[s], [b]=[t]\},\]
 and the {\em Peierls barrier} $h:M\times\mS^1\times M\times\mS^1\to\R$ by
\begin{equation}
  \label{eq:barrier}
 h(x,[s],y,[t]):=\liminf_{\lb b-a\rb\to\infty}
\bigl(F_{a,b}(x,y)+c(L)(b-a)\bigr)_{[a]=[s], [b]=[t]} 
\end{equation}
 We have $-\infty<\Phi\le h<\infty$.
The Lagrangian is called regular if the $\liminf$ in
\eqref{eq:barrier} is a $\lim$ and in that case, for each $s,t\in\R$
the convergence of the sequence $(F_{a,b})_{[a]=[s], [b]=[t]}$ is uniform. 

The critical value is the unique number $c$ such that \eqref{HJT}
has viscosity solutions $u:M\times\mS^1\to\R$. 
In fact (\cite{CIS}), for any $p$ the functions $z\mapsto h(p,z)$ 
and $z\mapsto-h(z,p)$ are respectively 
{\em backward} and {\em forward} viscosity solutions of \eqref{HJT}. 
Set $c=c(L)$ and let $\cS^- (\cS^+)$ be the set of {\em backward}
({\em forward}) viscosity solutions of \eqref{HJT}.
A subsolution of \eqref{HJT} always means a viscosity subsolution.

For a curve $\ga:I\to M$ we denote $\gab(t)=(\ga(t),[t])$
and $\Ga(t)=(\ga(t),\dga(t),[t])$. We say that $\ga:I\to M$ 
{\em calibrates} a subsolution
$u:M\times\mS^1\to\R$ of \eqref{HJT} if 
\[u(\gab(b))-u(\gab(a))=A_{L+c}(\ga|[a,b])\]
for any $[a,b]\subset I$. In such a case, $u$ is 
differentiable at $\gab(t)$ for any $t\in I^0$. 
Moreover $(\ga(t),Du(\gab(t)),[t])$ is
on the Hamiltonian orbit projecting on $\ga(t)$.  

If $u\in\cS^- (\cS^+)$, for any 
$(x,[s])\in M\times\mS^1$ there is $\ga:]-\infty,s]\to M$
($\ga:[s,\infty[\to M$) that calibrates $u$ and $\ga(s)=x$.
 
A pair $(u_-,u_+)\in \cS^-\times\cS^+$
is called {\em conjugated} if $u_-=u_+$ on $\cM=\pi(\tcM)$.
For such a pair $(u_-,u_+)$, we define $I(u_-,u_+)$ as the set where
$u_-$ and $u_+$ coincide. 

\begin{proposition}\label{inq}
Let $(x,v,[s])\in\tcM$ and $u$ be a viscosity subsolution of
\eqref{HJT}, then for any $t\leq\tau$ we have
\[u(\gab(\tau))-u(\gab(t))=A_{L+c}(\ga|[t,\tau]) \quad
 \forall t\leq\tau. \]
where $\Ga(t)=\phi_{t-s}(x,v,[s])$ 
\end{proposition}
\begin{lemma}\label{dif-I}
If $(x,[s])\in I(u_-,u_+)$, there is $\ga:\R\to M$ such that
$\ga(s)=x$ and
\[u_{\pm}(\gab(\tau))-u_{\pm}(\gab(t))=A_{L+c}(\ga|[t,\tau]) \quad
 \forall t\leq\tau. \]
It follows that $u_{\pm}$ is differentiable at $(x,[s])$ and
$Du_{\pm}(x,[s])=L_v(x,v,[s])$. 
\end{lemma}
Let
\[I^*(u_-,u_+)=\{(x,Du_\pm(x,[s]),[s]): \,(x,[s])\in I(u_-,u_+)\}.\]
We may define the {\em Aubry set} either as the set \cite{B} 
\[\cA^*:=\bigcap\{I^*(u_-,u_+):
(u_-,u_+)\text{ conjugated}\}\subset T^*M\times\mS^1\]
or as its pre-image under the Legendre transformation \cite{F}
\[\tcA:=\{(x,H_p(x,p,t),[t]):(x,p,[t])\in\cA^*\}.\]
 The projection of either Aubry set in $M\times\mS^1$ is
 \[\cA=\{\,(x,[t])\in M\times\mS^1:h(x,[t],x,[t])=0\}.\]
Consider the natural projections
$\Pr:T^*M\times\mS^1\to T^*M$, $\pi^*:T^*M\to M$. 
For $u$ a subsolution of (\ref{HJT}), we define
\[I^*(H,u)=\{(x,p,[t]):\ga(s):=\pi^*\circ\Pr\circ\phi^*_{s-t}(x,p,[t])
\text{  calibrates $u$ in }\R\} \] 
 \begin{lemma}\label{u-u+}
If $u:M\times\mS^1\to\R$ is a subsolution of \eqref{HJT},
there is a conjugated pair $(u_-,u_+)$ such that $u_+\le u\le u_-$.
\end{lemma}
We may take
\begin{align*}
u_-(x,[s])&=\min_{(y,[t])\in
  M\times\mS^1}\{u(y,[t])+h((y,[t]),(x,[s]))\} \\
u_+(x,[s])&=\max_{(y,[t])\in
  M\times\mS^1}\{u(y,[t])-h((x,[s]),(y,[t]))\}.
\end{align*}
\begin{proposition}\label{Hu}
If $u:M\times\mS^1\to\R$ is a subsolution of \eqref{HJT} and
$(u_-,u_+)$ is a conjugated pair such that $u_+\le u\le u_-$, then
$I^*(u_-,u_+)=I^*(H,u)$. 
\end{proposition}
\begin{proposition}\label{aubryconj}
There is a conjugated pair $(u_-,u_+)$ such that 
$\tcA=I(u_-,u_+)$ and $\cA^*=I^*(u_-,u_+)$.
\end{proposition}
It follows from Proposition \ref{Hu} that for a conjugated pair
$(u_-,u_+)$ given by Proposition \ref{aubryconj}  we have
$\cA^*=I^*(H,u_-)=I^*(H,u_+)$.  

In this article we assume the Aubry set $\cA^*$ is the union of the
hyperbolic periodic orbits $\Ga_i^*(t)$, $i\in[1,m]$
of the Hamiltonian flow with periods $N_i, i\in[1,m]$.

As shown in \cite{CIS} viscosity solutions are completely
determined by taking one value in each 
projected orbit $\gab_i(t)=(\ga_i(t),[t])$. Fix for instance 
$\xb_i=(x_i,[0])=\gab_i(0)$.
Because of the general properties recalled above, 
if $\phi$ is a viscosity solution of such that
$\phi(\xb_i) = \phi_i$ for all $i\in[1, m]$, we must have 
$\phi_j-\phi_i\le h(\xb_i,\xb_j)$ for all $i,j$. Conversely,
if this necessary condition is satisfied, then there is a unique
$\phi\in\cS^+$ having these prescribed values. In fact it is given by
\[\phi(z)=\max_i\phi_i-h(z,\xb_i).\]
Because the orbit $\Ga_i$ is hyperbolic, we have that $h_i(z)=h(z,\xb_i)$ 
is $C^2$  in a neighbourhood of the projected orbit.  A proof of 
this fact is similar to that of Proposition \ref{local} below. 
Define
\begin{align} \label{eq:average}
\lam_i:&=\frac 1{N_i}\int_0^{N_i} \Delta h_i(\gab_i(t))dt  \\
  \label{eq:minla}
\lab:&=\min_{1\in[1,m]} \lam_i.
\end{align}
Our main result is 
\begin{main}
Let $M=\T^d$ and assume $H$ satisfies \eqref{growth}.
Suppose $\phi_{\ep_n}$ converges to $\phi_0$ for a sequence
$\ep_n\to 0$, then
\[\phi_0(z)=\max\{\phi_0(\xb_i)-h_i(z):\lam_i=\lab\}.\]
In particular, if there is just one orbit $\ga_I$ such that $\lam_I=\lab$, 
then the solutions $\phi_{\ep}$ of\eqref{HJB}, normalized by 
$\phi_{\ep}(x_I,0)=0$, converge uniformly to $-h_I$ as $\ep\to 0$.
\end{main}
\begin{example}
 For $k\in\N$ let $V:\R\to\R$ be a $1/k$-periodic function with 
non-degenerate maxima $0\le x_1,\ldots,x_N<\dfrac 1k$,
$\max V=0$. Consider the mechanical Hamiltonian and Lagrangian
\[H_a(x,p)=\frac{p^2}2+V(x),\quad L_a(x,v)=\frac{v^2}2+V(x)\]
whose  projected Aubry set consists of the hyperbolic fixed points 
\[y_{ij}=x_i-\frac jk,i=1,\ldots,N,j=0,\ldots,k-1.\]
Let $h_a$ be the Peierls barrier for $H_a$, then 
$\phi(x)=-h_a(x,x_i)$
is a viscosity solution of 
\begin{equation}\label{mec}
\frac{\phi'^2}2+V(x)=0
\end{equation}
that is a $C^k$ in a neighbourhood of $x_i$. Differentiating \eqref{mec} twice
\[\phi''(x_i)^2+V''(x_i)=0.\]
Consider now the time periodic Hamiltonian 
\[H(x,p,t)=\frac{p^2}2-\frac pk+V(x+\frac tk)\]
with corresponding Lagrangian
\[L(x,v,t)=L_a(x+\frac tk,v+\frac 1k)\]
The projected Aubry set consists of the hyperbolic $k$-periodic orbits
$\gab_i(t)=(x_i-\dfrac tk,t), i=1,\ldots,N$. Function
$u(x,t)=\phi(x+\frac tk)$ is a viscosity solution of 
\[u_t+\frac{u_x^2}2-\frac{u_x}k+V(x+\frac tk)=0\]
which has a maximum at $\gab_i(t)$. From Lemma
\ref{root} below there is a neibourhood of $\gab_i$
where $u(x,t)=-h(x,t,\xb_i)$ and we have
\[u_{xx}(\ga_i(t))= \phi''(x_i)=-\sqrt{-V''(x_i)}.\]
 \end{example}

\section{Critical subsolutions}
An important tool for the proof of THEOREM is the existence of strict 
$C^k$ critical subsolutions in our setting, that extends the result  
of Bernard \cite{B} for the autonomous case
\label{sec:critical}
\begin{theorem}\label{subcrit}
Let $H:T^*M\times\R\to\R$ be $C^k$ satifying the standard hypotheses.
Assume the Aubry set $\cA^*$ is the union of a finite number of hyperbolic 
periodic orbits $\Ga_i^*$ of the the Hamiltonian flow, then there is
a $C^k$ subsolution $u$ of \eqref{HJT} such that 
\[u_t+H(x,Du(x,[t]),t)< c\]
for any $(x,[t])\notin\cA$.
\end{theorem}
Denote by $W_i^-, W_i^+$ the weak stable and unstable manifolds of $\Ga_i^*$
and $E^\pm(t)=d\Pr_{\Ga_i^*(t)}(T_{\Ga_i^*(t)}W^\pm)$.

For $\te\in T^*M$,
define the vertical subspace of $T_\te T^*M$ as
$V(\te)=\ker(d\pi^*(\te))$. If $H(\te)$ is the horizontal
subspace given by a riemannian  conexion in $M$, 
$T_\te T^*M=H(\te)\oplus V(\te)$. Let 
$h_\te:T_\te T^*M\to H(\te)$ be the corresponding projection.
Writting $\psi_s^t(\te)=\Pr\circ\phi^*_{t-s}(\te,[s])$ we say that
the points $(\te,[s])$, $\phi^*_{t-s}(\te,[s]), [t]\ne [s]$ 
are conjugated if $d{\psi_s^t}_\te V(\te)\cap V(\psi_s^t(\te))\ne\{0\}$. 
Being a minimizer, $\Ga_i^*$ has no pair of conjugated points and
writing $\te_i=\Pr\circ\Ga^*_i$, we have from \cite{CI} that   
$\forall R>0$ $\exists T(R,\te_i(s))>0$ such that 
$\|h_{\te_i(t)}d{\psi_s^t}_{\te_i(s)}w\|>R\|w\|$ for all $|t|>T(R,\te_i(s))$ and 
$w\in V(\te_i(s))\entre\{0\}$. 
Thus $V(\theta_i(s))\cap E_i^\pm(s)=\{0\}$ 
and so $V(\theta_i(s))\times\{0\}\cap T_{\Ga_i(s)}W_i^\pm=\{0\}$.
Then $W_i^\pm$ is locally the graph of a  $C^{k-1}$ transformation. 
\begin{proposition}\label{local}
The functions $u_\pm\in S^\pm$ given by Proposition \ref{aubryconj} are 
$C^k$ in a neighborhood of $\cA$.
\end{proposition}
\begin{proof}
Recall that $u_\pm$ is diferentiable at $(x,[s])$ iff there is only
one curve $\zeta:(-\infty,s]\to M$ calibrating $u_\pm$ with $\zeta(s)=x$
and moreover $Du_\pm(x,[s]) =L_v(\dot\zeta(s),s),[s])$.
We have the following Lemma for $u_-$, and its analogue for $u_+$
\begin{lemma}\label{vecindad}
Let $U^*$ be a compact neighborhood of $\Ga_i^*$
such that $U^*\cap\cA^*=\Ga_i^*$. There exists a neighborhood $U$
of $\gab_i$ such that for any curve $\zeta\in C^2((-\infty,s], M), s\in[0,1]$ 
calibrating $u_-$ with $\bar\zeta(s)\in U$, one has
$Z^*(t)=(\zeta(t),L_v(Z(t)),[t])\in U^*$ for all $t\le s$.
This implies that $Z^*(s)\in W_i^+$ and $u$ is differentiable at $\bar\zeta(s)$. 
\end{lemma}
\begin{proof}[Proof of Lemma \ref{vecindad}]
If the conclusion is not true, there is a sequence
of calibrating curves $\zeta_n\in C^2((-\infty,s_n], M) , s_n\in[0,1]$ with
$\lim\limits_{n\to\infty}d(\bar\zeta_n(s_n),\gab_i)=0$, and a sequence
of times $T_n\le s_n$ with  $Z_n^*(T_n)\in\partial U^*$.
Define $\Upsilon_n:(-\infty,s_n-\lb T_n\rb)\to T^*M\times{\mS}^1$ by
$\Upsilon_n(t)=Z^*_n(t+\lb T_n\rb)$. We can assume by taking a subsequence that 
$Z_n^*(T_n)$ converges to $(y,w,[\tau])$ and $\Upsilon_n$ converges
uniformly on compact sets to a limit  trajectory $\Upsilon:I \to T^*M\times{\mS}^1$, 
where the interval $I$ is either of the form $(-\infty,T]$  or $\R$.
We have $\Upsilon_n(\tau)\in\partial U^*$.
Since the curves $\pi^*\circ\Pr\circ \Upsilon_n$ calibrate $u_-$, 
so does $\upsilon=\pi^*\circ\Pr\circ \Upsilon$. 
If $I=(-\infty,T]$, then $\bar\upsilon(T)\in\gab_i$ since
$\lim\limits_{n\to\infty}d(\bar\zeta_n(s_n),\gab_i)=0$. 
Since $u_-$ is diferentiable on $\gab_i$, that is the only
calibrating curve at $\gab_i(T)$, so
$\Upsilon (T)=(\upsilon(T), Du_-(\bar\upsilon(T),[T])=\Ga_i^*(T)$
so $\Upsilon (T)\in\cA^*$, a contradiction.
If $I=\R$, then $\Upsilon(\tau)\in I^*(H,u)=\cA^*$, a contradiction again.
\end{proof}
From Lemma \ref{vecindad} there is a neighborhood $U$ of $\gab_i$ such that 
$u_-$ is differentiable in $U$ and graph$(Du|U)=(\pi^*\times Id)^{-1}(U)\cap W_i^+$.
\end{proof}
To get Theorem \ref{subcrit} the main tool is the following Lemma proved by
D. Masart 
\begin{lemma}\cite{Ma}
  There is a $C^2$ non-negative function $W:M\times\mS^1\to\R$,
  positive outside $\cA(L)$ and zero inside  $\cA(L)$ such that $c(L-W)=c(L)$
and  $\cA(L-W)=\cA(L)$.
\end{lemma}
It follows that for any function $V:M\times\mS^1\to\R$, 
positive outside $\cA(L)$ and such that $0\le V\le W$ we have 
$c(L-V)=c(L)$ and  $\cA(L-W)=\cA(L)$.
Function $V$ can be chosen so flat on $\cA(L)$ that
the linearized Hamiltonian flow along the orbits $\Ga_i^*$ is the
same for $H$ and $H+V$ . As a consequence, the orbits $\Ga_i^*$  remain
hyperbolic as orbits of $H + V$. Applying Proposition \ref{local} to
$H+V$ we obtain a solution of the Hamilton-Jacobi equation
\[u_t+H(x,Du,t)+V(x,t)=c(L-V)=c(L).\]
which is $C^k$ in a neighbourhood of $\cA(L-W)=\cA(L)$. This function
is a subsolution of \eqref{HJT} which is strict outside $\cA(L)$
and can be regularized to a $C^k$ subsolution of \eqref{HJT}.
Smoothing in this kind of context has been done in \cite{F}, Theorem 9.2. 
\section{Uniform Derivate bounds and Semiconvexity}
\label{sec:semicon}
\begin{lemma}\label{lip}
The periodic solutions $\phi_\ep$ of \eqref{HJB} have derivative uniformly
bounded.
 \end{lemma}
\begin{proof}
We first observe that  
\begin{equation}\label{cotcep}
\inf_{(x,t)} H(x,0,t)\leq c(\ep)\leq\sup_{(x,t)} H(x,0,t).
\end{equation}
Indeed, if $\phi_\ep$ has a minimum at $(\xb,\tb)$,  
\[D\phi_\ep(\xb,\tb)=0, d_t\phi_\ep(\xb,\tb)=0, \Delta\phi_\ep(\xb,\tb)\ge 0,\] 
so
\[c(\ep)=\ep\Delta\phi_\ep(\xb,\tb)+H(\xb,0,\tb)\ge\inf_{(x,t)} H(x,0,t).\] 
and similarly for the other inequality. 

We first prove that $D\phi_\ep$ is uniformly bounded.
To do that we follow \cite{bar} defining $w_\ep$ and $\psi_\ep$ by
\[\exp (w_\ep)=\psi_\ep=\max\phi_\ep-\phi_\ep+1,\]
and proving that for $K$ given in hypothesis (5), $|Dw_\ep|\leq K$.

Assuming we have this bound we prove that $\psi_\ep$ is uniformly bounded. 
Since $D\phi_\ep=-\psi_\ep Dw_\ep$, we get that
$D\phi_\ep$ is uniformly bounded.

Let $(x_\ep,t_\ep)$ be a point where $\phi_\ep$ attains its maximum.
Then $ w_\ep(x,t_\ep)\le Kd(x,x_\ep)\le K$  
and so $\psi_\ep(x,t_\ep)\le e^K$ for any $x\in\T^ d$. 
That $\psi_\ep$ is uniformly bounded follows from the maximum
principle. Indeed, let
\begin{align*}
  a&=\max_{(x,t)}H(x,0,t)-\min_{(x,t)} H(x,0,t),\\
v(x,t)&=\psi_\ep(x,t+t_\ep)-e^K+(a+\de)t.
\end{align*}
Since $v(x,0)\le 0$, if $v(x_0,t_0)>0$ for some $x_0,t_0<0$, there is a time
$\tb\in[t_0,0]$ and a point $\xb\in\T^d$ where $v(\xb,\tb)=0$ and
\[v_t(\xb,\tb)\le 0,\; D\phi_\ep(\xb,\tb+t_\ep)=Dv(\xb,\tb)=0,\; 
\Delta v(\xb,\tb)\le 0
\] 
Then
\begin{align*}
0\ge v_t(\xb,\tb) &=-c(\epsilon)-\ep\Delta v(\xb,\tb+t_\ep)+
H(\xb,0,\tb+t_\ep)+a+\de \\ 
   &\ge-\max_{(x,t)}H(x,0,t)+\min_{(x,t)} H(x,0,t)+a+\de\\
   &=\de.
\end{align*}
This contradiction shows that $v(x,t)\le 0$. Since $\de$ is arbitrary
and $\psi_\ep$ is periodic 
\[\psi_\ep(x,t)\le\max_{(x,t)}H(x,0,t)-\min_{(x,t)}H(x,0,t) +e^K.\]
Functions $w_\ep$ satisfy
\begin{equation}\label{ecw}
-d_tw_\ep-\ep|Dw_\ep|^2-\ep\Delta w_\ep+b(x,t,w_\ep,Dw_\ep)=0.
\end{equation}
where $b(x,t,u,p)=\exp(-u)[H(x,-\exp u\cdot p,t)-c(\ep)]$.

We prove $|Dw_\ep|\leq K$ using Bernstein's method, so let
$f=|Dw_\ep|^2$ and compute
\begin{align}  \label{bernstein} 
f_t&=2d_tDw_\ep Dw_\ep\\
\label{bernstein1}  Df&=2D^2 w_\ep Dw_\ep \\
  \Delta f&=2|D^2 w_\ep|^2+2D(\Delta w_\ep) Dw_\ep.
  \label{bernstein2}
\end{align}
Differentating \eqref{ecw} respect to $x$, multiplying by
$Dw_\ep$ and using 
\eqref{bernstein} \eqref{bernstein1} \eqref{bernstein2}
\begin{align*}
d_tDw_\ep Dw_\ep+DfDw_\ep+\ep D(\Delta w_\ep) Dw_\ep
 -b_x Dw_\ep-b_u f-b_p D^2 w_\ep Dw_\ep&=0\\
\frac 12f_t +DfDw_\ep+\frac{\ep}2\Delta f -\ep|D^2 w_\ep|^2
-b_x Dw_\ep-b_u f-\frac 12 b_pDf& =0
\end{align*}
If $f$ attains its maximum at $z_0=(x_0,t_0)$, then
\[f_t(z_0)=0, Df(z_0)=0, \Delta f(z_0)\le 0 \]
and so
\begin{equation}  \label{neg}
  \ep|D^2 w_\ep(z_0)|^2+b_x Dw_\ep(z_0)+b_u f(z_0)\le 0.
\end{equation}
We have
\begin{align*}
  b_x(x,t,u,p)&=\exp(-u)H_x(x,-\exp u\cdot p,t)\\
b_u(x,t,u,p)&=\exp(-u) [(H_p\cdot p-H)(x,-\exp u\cdot p,t)+c(\ep)].
\end{align*}
From (\ref{cotcep}) we get
\[b_u\ge\exp(-u)[(H_p\cdot p-H)(x,-\exp u\cdot p,t)+\inf H(x,0,t)].\]
Assume $|Dw_\ep(z_0)|>K$,  then $\exp w_\ep(z_0)|Dw_\ep(z_0)|\ge K$ 
since $w_\ep\ge 0$. From hypothesis (5), 
$b_u(z_0,w_\ep(z_0),Dw_\ep(z_0))\ge 0$ and
\begin{align*}
b_xDw_\ep(z_0)+b_uf(z_0)&>|Dw_\ep(z_0)| (b_uK-|b_x|)\\
&\ge\exp(-w_\ep(z_0)) |Dw_\ep(z_0)| \\
\cdot[K(H_p\cdot p-H+\inf H(x,0,t))&-|H_x|](x_0,-\exp w_\ep(z_0)Dw_\ep(z_0),t_0)\\
&\ge 0
\end{align*}
which contradicts \eqref{neg}.

We use again Bernstein's method to prove that $d_t\phi_\ep$,  
is uniformly bounded, so let $g=d_t\phi_\ep^2+|D\phi_\ep|^2$ and compute
\begin{align}  \label{eq:bernstein} 
g_t&=2 d_t\phi_\ep d_{tt}\phi_\ep +2d_tD\phi_\ep D\phi_\ep\\
\label{eq:bernstein1}  Dg&=2 d_tD\phi_\ep d_t\phi_\ep +2D^2\phi_\ep D\phi_\ep \\
  \Delta g&=2|d_tD\phi_\ep|^2+2d_t(\Delta\phi_\ep)+
2|D^2\phi_\ep|^2+2D(\Delta\phi_\ep) D\phi_\ep.
  \label{eq:bernstein2}
\end{align}

Differentating \eqref{HJB} first respect to $t$ and
multiplying by $d_t\phi_\ep$, then respect to  $x$ and multiplying by
$D\phi_\ep$, adding the results and using 
\eqref{eq:bernstein} \eqref{eq:bernstein1} \eqref{eq:bernstein2}

\begin{align*}
d_t\phi_\ep d_{tt}\phi_\ep+d_tD\phi_\ep D\phi_\ep
+\ep d_t(\Delta\phi_\ep)d_t\phi_\ep
+\ep D(\Delta\phi_\ep) D\phi_\ep&\\
  +H_td_t\phi_\ep+H_x D\phi_\ep+H_p
(d_tD\phi_\ep d_t\phi_\ep +D^2\phi_\ep D\phi_\ep) &=0\\
\frac 12g_t +\frac{\ep}2\Delta
g+H_td_t\phi_\ep+H_x D\phi_\ep+\frac 12 H_p
Dg-\ep|D^2\phi_\ep|^2-\ep |d_tD\phi_\ep|^2& =0
\end{align*}
If $g$ attains its maximum at $z_0\in\T^{d+1}$, then
\[g_t(z_0)=0, Dg(z_0)=0, \Delta g(z_0)\le 0. \]
If $C$ is a bound for $H_t(x,D\phi_\ep,t)$, 
$H_x(x,D\phi_\ep,t) D\phi_\ep$ and
$H(x,D\phi_\ep,t)-c(\ep)$, then at the point $z_0$ we have
\begin{align*}
\ep(\Delta\phi_\ep)^2&\le d\ep|D^2\phi_\ep|^2\le
d(H_td_t\phi_\ep+H_x D\phi_\ep)\\
&\le dC(|d_t\phi_\ep|+1)\le
dC(|c(\ep)-\ep\Delta\phi_\ep-H(x,D\phi_\ep,t)|+1)\\
&\le dC|\ep\Delta\phi_\ep|+dC^2+dC
\end{align*}
Thus, the values $\ep\Delta\phi_\ep(z_0)$, $d_t\phi_\ep(z_0)$ and 
$g(z_0)$ are bounded.
\end{proof}

The solution to the viscous equation \eqref{HJB} can be
characterized by a variational formula. We need to introduce a
probability space $(\Om,{\mathcal B},\PP)$ endowed with a brownian
motion $W(t):\Omega\to \T^d$ on the flat $d$-torus. We denote by $\E$ 
the expectation defined by the probability measure $\PP$. 

The solution to equation \eqref{HJB} satisfies Lax's formula
\begin{equation}
    \label{eq:lax0}
    \phi_\ep(x,t)=\sup_v \E\Bigl(\phi_\ep(X_\ep(\tau),\tau)-\int_t^\tau
L(X_\ep(s),v(s),s)ds+c(\ep)(t-\tau)\Bigr),
  \end{equation}
where $v$ is an admissible progressively measurable control process,
$\tau$ is a bounded stopping time
and $X_\ep$ is the solution to  the stochastic differential equation
\begin{equation}
  \label{eq:stoc}
\begin{cases}
  dX_\ep(s) &=v(s)ds+\sqrt{2\ep}\,dW(s)\\
X_\ep(t)&=x.
 \end{cases}
\end{equation}
See \cite{Flem} Lemma IV 3.1.

\begin{lemma}\label{semi}
The periodic solutions  $\phi_\ep $ of \eqref{HJB} are
uniformly semiconvex in the spatial variable. 
\end{lemma}
\begin{proof}
We have the following  description of the optimal $v$,
see for example \cite{Flem} Theorem IV 11.1.: 
Introduce the time dependent vector field
$U_\ep(x,t)=H_p(x,D\phi_\ep(x,t),t)$ and consider the solution $X_\ep(s)$
of the stochastic differential equation
\begin{equation}  \label{eq:optistoc}
\begin{cases}
  dX_\ep(s) &=U_\ep(X_\ep(s),s)ds+\sqrt{2\ep}\,dW(s)\\
X_\ep(t)&=x,
 \end{cases}
\end{equation}
then an optimal control in \eqref{eq:lax0} is given by the formula 
$v(s)=U_\ep(X_\ep(s),s)$.
Let $x\in\T^d$, $t\in[-2,-1]$ and take that optimal control, then 
\[ \phi_\ep(x,t)=
\E\bigl(\phi(X_\ep(0)+\int_t^0L(X_\ep(s),U_\ep(X_\ep(s),s),s)ds\bigr)+c(\ep)t\]

Let $|y|<1$ be an increment, the controls $U_\ep(X_\ep(s),s)\pm\dfrac yt$
are admissible and then
\[\phi_\ep(x\pm y,t)\ge\E\bigl(\phi(X_\ep(0)-
\int_t^0L(X_\ep(s)\pm\frac{sy}t,U_\ep(X_\ep(s))\pm\frac yt,s)ds\bigr)+c(\ep)t\] 
Let
$$M=1+\sup\limits_{\ep\in(0,1]}|U_\ep(x,s)|,$$
This is finite by Lemma \ref{lip}. Define now
$$A=\sup\limits_{|v|\le M}\|D_{xx}L(x,v,s)\|,
B=\sup\limits_{|v|\le M}\|D_{xv}L(x,v,s)\|,
C=\sup\limits_{|v|\le M}\|D_{vv}L(x,v,s)\|.$$
An application of Taylor's Theorem gives
\[L(x+\frac{sy}t,v+\frac yt,s)-2L(x,v,s)+L(x-\frac{sy}t,v-\frac yt,s)\le
A|\frac{sy}t|^2+2Bs|\frac yt|^2+C|\frac yt|^2 \]
for $|v|\le M-1$. Therefore
\begin{align*}
\phi_\ep(x+y,t) -2\phi_\ep(x,t) +\phi_\ep(x-y,t) &\ge
 -\int_t^0(As^2+2Bs+C)|\frac yt|^2 ds\\
&\ge\bigl(\frac{At}3-B+\frac Ct\bigr)|y|^2.
\end{align*}
\end{proof}

We will need the following

\begin{lemma}\label{convergeC1}
Suppose  the sequence $\phi_{\ep_n}$ of solutions of
\eqref{HJB}
converges uniformly to $\phi_0$. Assume that $\phi_0$ is differentiable 
in an open neighborhood $V$ of a periodic orbit. 
Then $D\phi_{\ep_n}$ converges to $D\phi_0$ 
uniformly in every compact subset of $V$.
\end{lemma}
This is an easy consequence of Lemma \ref{semi} and next
theorem, which is a slight extension of Theorem 25.7 in \cite{Roc}
and follows the same proof
\begin{theorem}
Let $A\subset\R^n$ be open convex, $B\subset\R^m$ be open, and 
$f_n:A\times B\to\R$ be a sequence of differentiable functions, convex in the 
first variable, converging uniformly to a differentiable function $f$.
Then $D_1f_n$ converges pointwise to $D_1f$, and in fact uniformly on
compact subsets.  
\end{theorem}

\section{Reduction to a regular Lagrangian}
\label{sec:regular}
In this section we show how to deduce THEOREM from the 
case when the Lagrangian is regular.

Let $N$ be the least common multiple of the periods of the
orbits $\Ga_1,\ldots,\Ga_m$ of the Aubry set. Define
\begin{align}\label{Levanta}
P_N:\T^d\times\R^d\times\mS^1 & \to \T^d\times\R^d\times\mS^1\\
                  (x,v,[t])  & \mapsto (x,\frac vN,[Nt]),\nonumber
\end{align}
and the Lagrangian $L_N=L\circ P_N$. The corresponding Hamiltonian is
given by 
\[H_N(x,p,t)=H(x,Np,Nt).\]
For a curve $\ga:[a,b]\to\T^d$ define
$\ga^N:[a/N,b/N]\to\T^d, t\mapsto\ga(Nt)$, then 
$NA_{L_N}(\ga^N)=A_L(\ga)$.
A curve $\ga$ is an extremal (minmizer)  of $L$ if and
only if the curve $\ga^N$ is an extremal (minimizer) of $L_N$.

Let $\ga^N_{i,j}(t)=\ga_i(Nt-j),j\in[1,N_i],i\in[1,m]$.
According to sections 3, 5 of \cite{B1}, the Aubry set of $L_N$ is the
union of the hyperbolic 1-periodic orbits 
$\Ga_{i,j}^N(t)=(\ga^N_{i,j}(t)\dga^N_{i,j}(t),[t])$ and $L_N$ is regular. 

For $\ep\ge 0$, a function $u:\T^d\times\R\to\R$ is a viscosity
solution of \eqref{HJB} if and only if $w(x,t)=\frac 1Nv(x,Nt)$ is a
viscosity solution of 
\begin{equation}
  \label{eq:hjlev}
  w_t+N\ep\Delta w+H_N(x,Dw,t)=c(\ep).
\end{equation}
\begin{lemma}\label{barreraN}
$h(x,[t],\xb_i)=N\min\limits_{j\in[1,N_i]}h_N(x,[\frac tN]),(x_i,[\frac jN])$
\end{lemma}
\begin{proof}
Since $L_N$ is regular 
\[h_N(x,[t],y,[s])=\lim_{n\to\infty}F^N_{t,s+n}(x,y)+c(0)(s+n-t)\]
Since 
  \[F_{t,j+nN}(x,x_i)+c(j+nN-t)=NF^N_{\frac tN,\frac jN+n}(x,x_i)+c(j+nN-t),\]
the sequence $(F_{t,n}(x,x_i)+c(0)(n-t))_{n\in\N}$ only accumulates to the values
\[Nh_N(x,[\frac tN],x_i,[\frac jN]), j\in[1,N].\]
Since
\[h_N(x_i,[\frac jN],x_i,[\frac{j+kN_i}N])=
A_{L_N+c(0)}(\ga_{ij}^N|_{[\frac jN,\frac{j+kN_i}N]})=\frac kNA_{L+c(0)}(\ga_i)=0\]
for $k\in\N$, there are in fact only $N_i$ accumulation values and then
\[h(x,[t],\xb_i) = \liminf _{n\to\infty}F_{t,n}(x,x_i)
=\min_{i\in[1,N_i]} Nh_N(x,[\frac tN],x_i,[\frac jN]), i\in[1,N].\]
\end{proof}
From Lemma \ref{barreraN}, in a
neighborhood of $(\ga_i(Nt-j),[t])$ we have 
\[Nh_N(x,[t],x_i,[\frac jN])=h(x,[Nt],\xb_i).\]
Then 
\begin{equation}
  \label{eq:lapN}
\lam^N_{ij}=\int_0^1\Delta h_N(\ga_i(Nt-j),[t],x_i,[\frac jN])dt
=\frac 1{N_i}\int_0^{N_i}\Delta h(\gab_i(t),\xb_i)dt=\lam_i. 
\end{equation}
\begin{proof}[Proof of THEOREM for $L$ assuming it holds for $L_N$]
Let $\phi_\ep$ be a periodic solution of \eqref{HJB} and suppose
$\phi_{\ep_n}$ converges to $\phi_0$ for a sequence $\ep_n\to 0$. 
The solutions $\psi_{\ep_n}(x,t)=\frac 1N\phi_{\ep_n}(x,Nt)$ of 
\eqref{eq:hjlev} converge to $\psi_0(x,t)=\frac 1N\phi_0(x,Nt)$. 
For $j\in[1,N], \psi_0(x,\frac jN)=\frac 1N\phi_0(x,0)$ and so
\[\psi_0(\gab^N_i(0))=\frac 1N\phi_0(\xb_i)-h_N(\gab^N_i(0),x_i,[\frac jN]).\] 
From \eqref{eq:lapN}, THEOREM for $L_N$ and Lemma \ref{barreraN},
\begin{align*}
\frac 1N\phi_0(x,Nt)&
=\max_{\lam_i=\lab}\max_{j\in[1,N_i]}\psi_0(\gab^N_i(0))-h_N(x,[t],\gab^N_i(0))\\
&=\max_{\lam_i=\lab}\max_{j\in[1,N_i]}
\frac 1N\phi_0(\xb_i)-h_N(x,[t],x_i,[\frac jN])\\
&=\frac 1N\max_{\lam_i=\lab}\phi_0(\xb_i)-h(x,[Nt],\xb_i).
\end{align*}
\end{proof}
\section{Regular Lagrangians}
\label{regular}
In this section we assume that the Lagrangian is regular.
Let $f:\T^{d+1}\to\R$ be a strict $C^k$ subsolution of \eqref{HJT} given by 
Theorem \ref{subcrit} and consider the Lagrangian
\[\mL(x,v,t)=L(x,v,t)-Df(x,[t])v-f_t(x,[t])+c\]
with Hamiltonian $\mH(x,p,t)=H(x,p+Df,[t])+f_t(x,[t])-c$. 
If $\af\in\cC(x,s,y,t)$, $A_{\mL}(\af)=A_{L+c}(\af)+f(x,[s])-f(y,[t])$.
Thus $L$ and $\mL$  have the same Euler Lagrange flow and projected
Aubry set and the Peierls barrier of $\mL$ is $h(z,w)-f(w)+f(z)$. Moreover,
\begin{equation}\label{mL} 
\forall (x,v,t)\; \mL(x,v,t)\ge 0,\; \tilde\cA=\{(x,v,t):\mL(x,v,t)=0\}
\end{equation}
and $u$ is a viscosity solution of \eqref{HJT}
if and only if $u-f$ is a viscosity solution of 
\begin{equation}\label{eq:mhj}
v_t+\mH(x,Dv,t)=0.  
\end{equation}
\begin{lemma}\label{root}
 Assume the Lagrangian also satisfies \eqref{mL}.
A $\phi\in\cS^+$ has a local maximum at $\gab_i$ if and only if
\begin{equation}
  \label{eq:root}
  \forall j\ne i\quad \phi(\xb_i) > \phi(\xb_j) -h(\xb_i,\xb_j).
\end{equation}
\end{lemma}
From the continuity of $\phi$ and $h$, if condition \eqref{eq:root}
holds, there is a neighbourhood of $\gab_i$ where
\[\phi = \phi(\xb_i)-h(\cdot,\xb_i) \]
\begin{proof}
 Let $V$ be a neighborhood where $\gab_i$  
is a maximum of $\phi$. Suppose that there is $j\ne i$ such that
\begin{equation}\label{eq}
\phi(\xb_i) = \phi(\xb_j) -h(\xb_i,\xb_j)
\end{equation}
Let $\ga_n:[0,n]\to\T^d$ be a curve joining $x_i$ to
$x_j$ such that
\[A_L(\ga_n)=F_{0,n}(x_i,x_j).\]
Let $t_n\in[0,n]$ be the first exit time of $\gab_n(t)$ out of $V$, and
$\gab_n(t_n)$ be the first point of intersection with $\partial U_j$. 
As $n$ goes to infinity, $t_n$ and $n-t_n$ tend to infinity. 
This follows from the fact that $\dga_n(0)$ has to tend to $\dga_i(0)$, 
and $\dga_n(n)$ has to tend to $\dga_j(0)$. To justify this, 
consider $v$ a limit point of $\dga_n(0)$, and $\ga:\R\to\T^d$
the solution to the Euler-Lagrange equation such that
$\ga(0)=x_i,\dga(t)=v$. From the fact that
\[F_{0,n}(x_i,x_j)-F_{1,n}(\ga_n(1),x_j)=A_L(\ga_n|_{[0,1]})\]
and the regularity of $L$, taking limit $n\to \infty$ it follows  
\[h(\xb_i,\xb_j)-h(\gab(1),\xb_j)=A_L(\ga|_{[0,1]}).\]
Since $\ga_i(-1)=x_i$ and $L=0$ on $\tcA$
\[h(\gab_i(-1),\xb_i)-h(\gab(1),\xb_j)=
A_L(\ga_i|_{[-1,0]})+A_L(\ga_{[0,1]})\]
so that the curve obtained by gluing $\ga_i|_{[-1,0]}$ with
$\ga|_{[0,1]}$ minimizes the action between its endpoints. In
particular, it has to be differentiable, thus $v=\dga(0)=\dga_i(0)$.
Let $(y,w,\tau)$ be a cluster point of 
$(\ga_n(t_n),\dga_n(t_n),t_n-\lb t_n\rb)$. 
From the fact
\[F_{0,t_n}(x_i,\ga_n(t_n))+A_L(\ga_n|_{[t_n,n]})=F_{0,n}(x_i,x_j)\]
and the uniform convergence of $F_{a,b}$ when $\lb b-a\rb\to\infty$, 
we obtain 
\[h(\xb_i,y,\tau)+h(y,\tau,\xb_j)= h(\xb_i,\xb_j).\]
Then
\begin{align*}
\phi(\xb_i)&\ge \phi(y,[\tau])\\
               &\ge \phi(\xb_j)-h(y,[\tau],\xb_j)\\
               &=\phi(\xb_j)-h(\xb_i,\xb_j)+h(\xb_i,y)\\
               &=\phi(\xb_i)+h(\xb_i,y,[\tau]).
\end{align*}
This contradiction shows that \eqref{eq} can not happen.
\end{proof}
\begin{corollary}\label{min-regu}
Assume the Lagrangian also satisfies \eqref{mL}.
Let $\phi\in\cS^+$ and
 $B=\{i:\gab_i\text{ is a local maximum of }\phi\}$. Then
\[\phi=\max_{i\in B}\phi(\xb_i) -h(\cdot,\xb_i).\]
\end{corollary}
\begin{proof}
For $z\in\T^{d+1}$ let $i$ be such that $\phi(z)=\phi(\xb_i)-h(z,\xb_i)$. 
If $i\notin B$
there is $j\ne i$ such that $\phi(\xb_i)=\phi(\xb_j) -h(\xb_i,\xb_j)$ and then 
\[\phi(\xb_j) -h(z,\xb_j)\ge\phi(\xb_j)-h(z,\xb_i)-h(\xb_i,\xb_j)
=\phi(\xb_i)-h(z,\xb_i)=\phi(z).\]
If $j\notin B$ there is $k\ne j$ such that 
$\phi(\xb_j)=\phi(\xb_k)-h(\xb_j,\xb_k)$ and then 
\[\phi(z)=\phi(\xb_k) -h(z,\xb_k),\quad
h(\xb_i,\xb_k) =h(\xb_i,\xb_j)+h(\xb_j,\xb_k).\]
Thus $k\ne i$. We continue until we arrive to $l\in B$ with
\[\phi(z)=\phi(\xb_l) -h(z,\xb_l).\] 
\end{proof}
We now take assumption \eqref{mL} out.
Recall that $f:\T^{d+1}\to\R$ is a strict $C^k$ subsolution of
\eqref{HJT}.
\begin{corollary}\label{min-rep}
Let $\phi\in\cS^+$,
$B=\{i:\gab_i\text{ is a local maximum of }\phi-f\}$.
\begin{itemize}
\item 
 For $i\in B$ there is a neighborhood of $\gab_i$ where
\[\phi=\phi(\xb_l) -h(\cdot,\xb_l).\]
\item 
For any $z\in\T^{d+1}$
\[\phi(z)=\max_{i\in B}\phi(\xb_i)-h(z,\xb_i).\]
\end{itemize}
\end{corollary}
\begin{proof}
Let $i\in B$ and apply Lemma \ref{root} to the Lagrangian $\mL$
to get a neighborhood of $\gab_i$ where
\[\phi-f=\phi(\xb_l) -f(\xb_i)-h(z,\xb_l)+f(\xb_i) -f.\] 
Applying Corollary \ref{min-rep} we have
\[\phi-f=\max_{i\in B}\phi(\xb_i) -f(\xb_i) -h(\cdot,\xb_i)+f(\xb_i)-f.\]
\end{proof}

 \begin{lemma}\label{l.estimacion.p}
 \[c'_+(0)=\liminf_{\ep\to 0^+}\frac{c(\ep)-c(0)}\ep\ge -\lab\]
 \end{lemma}
\begin {proof} We will prove that
\[\liminf_{\ep\to 0^+}\frac{c(\ep)-c(0)}\ep\ge -\lab-r\]
for an arbitrary $r>0$.
Take $I$ with $\lam_I=\lab$ and let $\Phi$ be a $C^3$ function that
coincides with $-h_I=-h(\cdot,\xb_I)$ in a neighbourhood $V$ of $\gab_I$.

Defining $U(x,t)=H_p(x,D\Phi(x,[t]),t)$,
we have that $\gab_I$ is a attractive periodic orbit of the vector field
$(U(x,t),1)$. Let $X_\ep$ be the solution to
\begin{equation}
  \label{eq:stocas}
\begin{cases}
  dX_\ep(t) &=U(X_\ep(t),t)dt+\sqrt{2\ep}\,dW(t)\\
X_\ep(0)&=x_I.
\end{cases}
\end{equation}
To easy notation we write $\Xb_\ep(t)=(X_\ep(t),t)$
Let $\de>0$ be sufficiently small to have
$\de{\|\Phi\|_{C^3}}\le{r}$ and
$B_\de(\gab_I):=\{(x,[t]):d(x,\ga_I(t))\leq\de\}\subset V$, and define
the stopping time 
\begin{equation}
   \label{eq:paro}
\tau(\om)=\min\{s>0:d(X_\ep(s,\om),\ga_I(s))\ge \de\}.
 \end{equation}
From \eqref{eq:lax0} and equalities
\begin{align*}
L(x,U(x,t),t)+H(x,D\Phi(x,[t]),t)&=D\Phi(x,[t])U(x,t) \\
\Phi_t+H(x,D\Phi(x,[t]),t)&=c(0)   \mbox{ for all }(x,[t])\in V,
\end{align*}
\begin{align*}
  (c(\ep)-c(0))\E(\tau\wedge \kappa)&\ge \\
\E\Bigl(\phi_\ep(\Xb_\ep(\tau\wedge \kappa))&-\phi_\ep(\xb_I)-
\int\limits_0^{\tau\wedge \kappa} \Phi_s+D\Phi(\Xb_\ep(s))U(\Xb_\ep(s))ds\Bigr) ,
\end{align*}
for all $\kappa>0$ (where $\tau\wedge \kappa$ denote
the bounded stopping time $\min(\tau, \kappa)$).

An application of Dynkin's formula gives
\[
\E(\Phi(\Xb_\ep(\tau\wedge \kappa))-\Phi(\xb_I)=\E\Bigl(\int\limits_0^{\tau\wedge \kappa} \Phi_s+D\Phi(\Xb_\ep(s))U(\Xb_\ep(s))ds
+\ep\Delta\Phi(\Xb_\ep(s))ds\Bigr).
\]
Defining $\psi_\ep=\phi_\ep-\Phi$ we get
\[(c(\ep)-c(0))\E(\tau\wedge \kappa)\ge
\E\Bigl(\psi_\ep(\Xb_\ep(\tau\wedge \kappa))-\psi_\ep(\xb_I)) 
+\ep\int\limits_0^{\tau\wedge \kappa}\Delta \Phi(\Xb_\ep(s))ds\Bigr).\]
For $s\in[0,\tau(\om)]$,
\[|\Delta\Phi(\Xb_\ep(s,\om))+\Delta
h_I(\gab_I(s))|\le\|\Phi\|_{C^3}\de\le r\]
so that
\[\left|\E\Bigl(\int\limits_0^{\tau\wedge \kappa}\Delta\Phi(\Xb_\ep(s))ds\Bigr)+
\E\Bigl(\int\limits_0^{\tau\wedge \kappa}\Delta h_I(\gab_I(s))ds\Bigr)\right|\le
\E(\tau\wedge \kappa)r.\] 
Let $M=\sup\limits_{x,\ep}|\psi_\ep(x)|$ (which is finite by Lemma
\ref{lip}), then 
\[\frac{c(\ep)-c(0)}{\ep}\ge-\frac{2M}{\ep \E(\tau\wedge
  \kappa)}-\frac 1{\E(\tau\wedge \kappa)} 
\E\Bigl(\int\limits_0^{\tau\wedge \kappa}\Delta h_I(\gab_I(s))ds\Bigl)-r.\]
\begin{align*}
\E\Bigl(\int\limits_0^{\tau\wedge \kappa}\Delta h_I(\gab_I(s))ds\Bigl)
&=\int_0^{\infty}
\Delta h_I(\gab_I(s))\PP(\tau\wedge \kappa >s)ds\\
&=\int_0^{N_I}\Delta h_I(\gab_I(s))
\Bigl(\sum_{k=0}^{\infty}\PP(\tau\wedge\kappa>s+kN_I)\Bigl)ds\\ 
&\le \int_0^{N_I}\Delta h_I(\gab_I(s))
\Bigl(1+\int_0^{\infty}\PP(\tau\wedge
\kappa>s+uN_I)du\Bigl)ds\\ 
&=\int_0^{N_I}\Delta h_I(\gab_I(s))\Bigl(1+\E(\frac{\tau\wedge
  \kappa-s}{N_I})\Bigl)ds 
\end{align*}
so that
$$\frac 1{\E(\tau\wedge \kappa)}
\E\Bigl(\int\limits_0^{\tau\wedge \kappa}\Delta h_I(\gab_I(s))ds\Bigl)\leq
\frac{1}{N_I}\int_0^{N_I}\Delta h_I(\gab_I(s))ds 
+\frac{N_I||h_I||_{C^2(V)}}{\E(\tau\wedge \kappa)}$$
and we can now let $\kappa$ tend to infinity to obtain:
\[\frac{c(\ep)-c(0)}{\ep}\ge-\frac{2M}{\ep
  \E(\tau)}-\frac{N_I||h_I||_{C^2(V)}}{\E(\tau)} 
-\frac{1}{N_I}\int_0^{N_I}\Delta h_I(\gab_I(s))ds-r.\]
Freidlin and Wentzel  (\cite{F-W}, Chapter 4.4) gave an estimate for
$E(\tau)$, for a stochastic perturbation of a vector field
having a sink. Although now the vector field has an attractive
periodic orbit $\gab_I$, the estimate of \cite{F-W} still
applies:
\[m=\liminf_{\ep\to 0} \ep\log \E(\tau) >0 .\]
Letting now $\ep>0$ go to zero we obtain \quad
$c'_+(0)\ge-\lam_I-r.$ 
\end{proof}
Suppose that a sequence $(\phi_{\ep_n})$ of solutions of $\eqref{HJB}$
converges to $\phi_0$. Let $\psi$ be a $C^3$ function that coincides
with $\phi_0$ on a neigbourhood $V_i$ of each $\gab_i$ that is
a local maximum of $\phi_0-f$ (such a function $\psi$ exists by Corollary
\ref{min-rep}). 
Then $\psi_\ep=\phi_\ep-\psi$ is a solution to the equation
\begin{equation}\label{modifham}
d_t\psi_\ep+\ep\Delta\psi_\ep+\tH_\ep(x,D\psi_\ep,t)=c(\ep),
\end{equation}
where
\[\tH_\ep(x,p,t)=d_t\psi+\ep\Delta\psi+H(x,D\psi_\ep+p,t)\]
with corresponding lagrangian
\[\tL_\ep(x,v,t)=L(x,v,t)-D\psi\cdot v-d_t\psi-\ep\Delta\psi.\]
As in \eqref{eq:lax0}, $\psi_\ep$ satisfies the variational
formulation of  Equation \eqref{modifham}:
\begin{equation}\label{eq:lax1}
    \psi_\ep(x,t)=\sup_u \E\Bigl(\psi_\ep(\Xb_\ep(\tau))-\int_t^\tau
\tL_\ep(X_\ep(s),u(s),s)ds+c(\ep)(t-\tau)\Bigr).
  \end{equation}
\begin{lemma}\label{l.igualdad.p}
If $\gab_i$ is a local maximum of the function $\phi_0-f$,
then $\lam_i=\lab$ and
 \[\lim_{n\to\infty}\frac{c(\ep_n)-c(0)}{\ep_n}=-\lab \]
\end{lemma}
\begin{proof} Let $2r=\min\{\lam_j-\lab:\lam_j>\lab\}$ and
consider 
\[u_\ep(x,t)=D_p\tH_\ep(x,D\psi_\ep(x,t),t)=H_p(x,D\phi_\ep(x,t),t).\]

Given $1\leq i\leq m$, let $X_\ep$ be the solution to
\begin{equation}
  \label{eq:stoca}
\begin{cases}
  dX_\ep(t) &=u_\ep(\Xb_\ep(t))dt+\sqrt{2\ep}\,dW(t)\\
X_\ep(0)&=x_i.
 \end{cases}
\end{equation}
We know that $u_\ep(\Xb_\ep(t))$ is the optimal control
associated to the variational formulation \eqref{eq:lax1}, which
means that, for all bounded stopping time $\tau$,
\begin{align*}
\psi_\ep(\xb_i)&=\E\Bigl(\psi_\ep(\Xb_\ep(\tau))-
\int_0^\tau\bigl(L(X_\ep(s),u_\ep(\Xb_\ep(s)),s)\\
&-D\psi(\Xb_\ep(s))u_\ep(\Xb_\ep(s))
-\ep\Delta\psi(\Xb_\ep(s))\bigr)ds-c(\ep) \tau\Bigr).
  \end{align*}

Let $\de>0$ be sufficiently small to have $\de{\|\psi\|_{C^3}}\le{r}$ and
$B_\de(\gab_i)\subset V_i$ and define
\begin{equation}
   \label{eq:paro1}
\tau(\om)=\min\{s>0:d(X_\ep(s,\om),\ga_i(s))\ge \de\}.
 \end{equation}
Since 
\begin{align*}
  L(x,u_\ep(x,t),t)+H(x,D\psi(x,t),t) &\ge D\psi(x,t)u_\ep(x,t)\\
  \psi_t+H(x,D\psi(x,t),t)&=c(0)  \text{ for }(x,t)\in V_i,
\end{align*}
\[((c(\ep)-c(0))\E(\tau\wedge \kappa)\le
\E\Bigl(\psi_\ep(\Xb_\ep(\tau\wedge \kappa))-\psi_\ep(\xb_i) 
+\ep\int\limits_0^{\tau\wedge \kappa}\Delta\psi(\Xb_\ep(s))ds\Bigr)\]
for all $\kappa>0$.

For $s\in[0,\tau(\om)]$,
\[|\Delta\psi(\Xb_\ep(s,\om))+\Delta
h_i(\gab_i(s))|\le\|\psi\|_{C^3}\de\le r\]
so that
\[\left|\E\Bigl(\int_0^{\tau\wedge \kappa}\Delta\psi(\Xb_\ep(s))ds\Bigr)+
\E\Bigl(\int_0^{\tau\wedge \kappa}\Delta h_i(\gab_i(s))ds\Bigr)\right|\le
\E(\tau\wedge \kappa)r.\]
Let $M=\sup\limits_{x,\ep}|\psi_\ep(x)|$, then
\[\frac{c(\ep)-c(0)}{\ep}\le\frac{2M}{\ep 
\E(\tau\wedge \kappa)}-\frac 1{\E(\tau\wedge \kappa)}
\E\Bigl(\int_0^{\tau\wedge \kappa}\Delta h_i(\gab_i(s))ds\Bigl)+r.\]
Reasoning as in the proof of Lemma \ref{l.estimacion.p}, we get
$$\frac 1{\E(\tau\wedge \kappa)}
\E\Bigl(\int_0^{\tau\wedge \kappa}\Delta h_i(\gab_i(s))ds\Bigl)\geq
\frac{1}{N_i}\int_0^{N_i}\Delta h_i(\gab_i(s))ds 
-\frac{2N_i||h_i||_{C^2(V_i)}}{\E(\tau\wedge \kappa)}$$
and we can now pass to the limit $\kappa\to +\infty$ to get
\[\frac{c(\ep)-c(0)}{\ep}\le\frac{2M}{\ep
  \E(\tau)}-\frac{1}{N_i}\int_0^{N_i}\Delta h_i(\gab_i(s))ds 
-\frac{2N_i||h_i||_{C^2(V_i)}}{\E(\tau)}+r.\]

By Lemma \ref{convergeC1},
$(u_{\ep_n})$ converges uniformly to 
$H_p(x,D\phi_0(x,t),t)$ in the
neighborhood $V_i$; the estimate of
Freidlin and Wentzell for $\E(\tau)$ also
applies (\cite{F-W}, Chapter 5.3):
\[m=\liminf_{n\to\infty} \ep_n\log \E(\tau) >0 ,\]
and so, letting $n$ grow we obtain
 \[\limsup_{n\to\infty}\frac{c(\ep_n)-c(0)}{\ep_n}\le-\lam_i+r,\]
which, by our choice or $r$, is possible only if $\lam_i=\lab$.
\end{proof}
Lemma \ref{l.igualdad.p} and Corollary \ref{min-rep} imply THEOREM for
a regular Lagrangian.

\end{document}